\documentclass[11pt]{article}
\usepackage[tbtags]{amsmath}
\usepackage{amssymb}
\usepackage{amsthm}
\usepackage[misc]{ifsym}
\usepackage{cases}
\usepackage{mathrsfs}
\usepackage{graphicx}
\usepackage{subcaption} 

\numberwithin{equation}{section}
\normalsize
\title{\bf Backward Linear-Quadratic Mean Field Stochastic Differential Games: A Direct Method \thanks{This work is supported by National Key R\&D Program of China (2022YFA1006104), National Natural Science Foundations of China (12471419, 12271304), and Shandong Provincial Natural Science Foundations (ZR2024ZD35, ZR2022JQ01)}}
\author{\normalsize  Yu Si\thanks{\it School of Mathematics, Shandong University, Jinan 250100, P.R. China, E-mail: 202112003@mail.sdu.edu.cn},\quad Jingtao Shi\thanks{\it Corresponding author. School of Mathematics, Shandong University, Jinan 250100, P.R. China, E-mail: shijingtao@sdu.edu.cn}}
\date{}
\newtheorem{Proposition}{Proposition}[section]
\newtheorem{Theorem}{Theorem}[section]

\newtheorem{Lemma}{Lemma}[section]
\newtheorem{Remark}{Remark}[section]
\newtheorem{Assumption}{Assumption}[section]
\newtheorem{Problem}{Problem}[section]
\begin{document}

\maketitle

\noindent{\bf Abstract:}\quad This paper studies a linear-quadratic mean-field game of stochastic large-population system, where the large-population system satisfies a class of $N$ weakly coupled linear backward stochastic differential equation. Different from the fixed-point approach commonly used to address large population problems, we first directly apply the maximum principle and decoupling techniques to solve a multi-agent problem, obtaining a centralized optimal strategy. Then, by letting $N$ tend to infinity, we establish a decentralized optimal strategy. Subsequently, we prove that the decentralized optimal strategy constitutes an $\epsilon$-Nash equilibrium for this game. Finally, we provide a numerical example to simulate our results.

\vspace{2mm}


\vspace{2mm}


\section{Introduction}

Recently, the study of dynamic optimization in stochastic large-population systems has garnered significant attention. Distinguishing it from a standalone system, a large-population system comprises numerous agents, widely applied in fields such as engineering, finance and social science. In this context, the impact of a single agent is minimal and negligible, whereas the collective behaviors of the entire population are significant. All the agents are weakly coupled via the state average or empirical distribution in dynamics and cost functionals. Consequently, centralized strategies for a given agent, relying on information from all peers, are impractical. Instead, an effective strategy is to investigate the associated {\it mean-field games} (MFGs) to identify an approximate equilibrium by analyzing its limiting behavior. Along this research direction, we can obtain the decentralized strategies through the limiting auxiliary control problems and the related {\it consistency condition} (CC) system. The past developments have largely followed two routes. One route starts by formally solving an $N$-agents game to obtain a large coupled solution equation system. The next step is to derive a limit for the solution by taking $N \rightarrow \infty$ \cite{Lasry-Lions-2007}, which can be called the direct (or bottom-up) approach. The interested readers can refer to \cite{Cong-Shi-2024, Huang-Zhou-2020, Wang-2024, Wang-Zhang-Zhang-2022}. Another route is to solve an optimal control problem of a single agent by replacing the state average term with a limiting process and formalize a fixed point problem to determine the limiting process, and this is called the fixed point (or top-down) approach. This kind of method is also called {\it Nash certainty equivalence} (NCE) \cite{Huang-Caines-Malhame-2007, Huang-Malhame-Caines-2006}. The interested readers can refer to \cite{Bardi-Priuli-2014, Bensoussan-Feng-Huang-2021, Feng-Wang-2024, Hu-Huang-Li-2018, Hu-Huang-Nie-2018, Huang-Huang-2017, Huang-Wang-2016, Moon-Basar-2017, Nguyen-Huang-2012, Wang-Zhang-2017}. Further analysis of MFGs and related topics can be seen in \cite{Bensoussan-Frehse-Yam-2013, Buckdahn-Djehiche-Li-Peng-2009, Carmona-Delarue-2013, Huang-2010, Li-Sun-Xiong-2019} and the reference therein. 

A {\it backward stochastic differential equation} (BSDE) is a {\it stochastic differential equation } (SDE) with a given random terminal value. As a consequence, the solution to BSDE consists of one adapted pair $(y(\cdot),z(\cdot))$. Here, the second component $z(\cdot)$ is introduced to ensure the adaptiveness of $y(\cdot)$. The linear BSDE was firstly introduced by \cite{Bismut-1978}. Then, \cite{Pardoux-Peng-1990} generalized the nonlinear case. In fact, mean-field problems driven by backward systems can be used to solve economic models with recursive utilities and cooperative relations (\cite{Du-Huang-Wu-2018, Feng-Huang-Wang-2021,  Huang-Wang-Wu-2016, Huang-Wang-Wang-Wang-2023, Li-Wu-2023}). One example is the production planning problem for a given minimum terminal, where the goal is to maximize the sum of product revenue. Another example is the hedging model of pension funds. In this case, we usually consider many types of pension funds and minimize the sum of the model risks. 

In this paper, we consider a class of {\it linear-quadratic} (LQ) mean-field games with backward stochastic large-population system. Compared with the existing literature, the contributions of this paper are listed as follows.
\begin{itemize}
  \item The LQ backward MFG is introduced to a general class of weakly-coupled backward stochastic system. The second part $z_{ij}(\cdot)$ of the solution of state equation is introduced to ensure the adaptiveness of $x_i(\cdot)$, which also enters the cost functional. 
  \item Different from the fixed-point approach commonly used to address large population problems, we adopts the direct approach to solve the problem. Apply the maximum principle to solve a multi-agent problem, the optimal centralized strategy can be represented via the Hamiltonian system and adjoint process. And we introduce some Riccati equations, an SDE and a BSDE to obtain linear feedback form of centralized strategies. As $N$ tend to infinity, we obtain the decentralized strategy. 
  \item We give numerical simulations of the optimal state and optimal decentralized strategy to demonstrate the feasibility of our theoretical results.
\end{itemize}

The rest of this paper is organized as follows. In Section 2, we formulate our problem. In Section 3, we design the decentralized strategy. In Section 4, we prove the decentralized optimal strategies are the $\epsilon$-Nash equilibria of the games. In Section 5, we give a numerical example. Finally, the conclusion is given in Section 6.

\section{Problem formulation}

Firstly, we introduce some notations that will be used throughout the paper. We consider a finite time interval $[0, T]$ for a fixed $T > 0$. Let $\left(\Omega, \mathcal{F}, \left\{\mathcal{F}_t\right\}_{t\geq0}, \mathbb{P}\right)$ be a complete filtered probability space, on which a standard $N$-dimensional Brownian motion $\left\{W_i(t), 1 \leq i \leq N\right\}_{t \geq 0}$ is defined, and $\{\mathcal{F}_t\}$ is defined as the complete information of the system at time $t$. That is, for any $t \geq 0$, we have
$$
\mathcal{F}_t:=\sigma\left\{W_i(s), 1 \leq i \leq N, 0 \leq s \leq t\right\}.
$$

Let $\mathbb{R}^n$ be an $n$-dimensional Euclidean space with norm and inner product being defined as $|\cdot|$ and $\langle\cdot, \cdot\rangle$, respectively. Next, we introduce three spaces. A bounded, measurable function $f(\cdot):[0, T] \rightarrow \mathbb{R}^n$ is denoted as $f(\cdot) \in L^\infty(0, T; \mathbb{R}^n)$. An $\mathbb{R}^n$-valued, $\mathcal{F}_t$-adapted stochastic process $f(\cdot): \Omega \times [0, T] \rightarrow \mathbb{R}^n$ satisfying $\mathbb{E} \int_0^T |f(t)|^2 dt < \infty$ is denoted as $f(\cdot) \in L_{\mathcal{F}}^2(0, T; \mathbb{R}^n)$. An $\mathbb{R}^n$-valued, $\mathcal{F}_T$-measurable random variable $\xi$ with $\mathbb{E} \xi^2 < \infty$ is denoted as $\xi \in L_{\mathcal{F}_T}^2(\Omega, \mathbb{R}^n)$.

For any random variable or stochastic process $X$ and filtration $\mathcal{H}$, $\mathbb{E}X$ represent the mathematical expectation of $X$. For a given vector or matrix \(M\), let \(M^{\top}\) represent its transpose. We denote the set of symmetric \(n \times n\) matrices (resp. positive semi-definite matrices) with real elements by \(\mathcal{S}^n\) (resp. \(\mathcal{S}_{+}^n\)). If \(M \in \mathcal{S}^n\) is positive (semi) definite, we abbreviate it as \(M > (\geq) 0\). For a positive constant \(k\), if \(M \in \mathcal{S}^n\) and \(M > kI\), we label it as \(M \gg 0\).

Now, let us focus on a large population system comprised of $N$ individual agents, denoted as $\left\{\mathcal{A}_i\right\}_{1 \leq i \leq N}$. The state $x_i(\cdot)\in \mathbb{R}^n$ of the agent $\mathcal{A}_i$ is given by the following linear $\mathrm{BSDE}$:
\begin{equation}\label{state}
\left\{\begin{aligned}
	d x_i(t)&=\left[A(t) x_i(t)+B(t) u_i(t)\right] d t+\sum_{j=1}^{N}z_{ij}(t) d W_j, \\
	x_i(T)&=\xi_i,
\end{aligned}\right.
\end{equation}
where $u_i(\cdot) \in \mathbb{R}^k$ is the control process of agent $\mathcal{A}_i$, and $\xi_i \in L_{\mathcal{F}_T}^2(\Omega, \mathbb{R}^n)$ represents the terminal state, the coefficients $A(\cdot)$, $B(\cdot)$ are deterministic functions with compatible dimensions. Noting that $\{z_{ij}(\cdot),1 \leq i,j \leq N\}$ are also the solution of (\ref{state}), which are introduced to to ensure the adaptability of $x_i(\cdot)$.

The cost functional of agent $\mathcal{A}_i$ is given by
\begin{equation}\label{cost}
\begin{aligned}
	\mathcal{J}_i\left(u_i(\cdot); u_{-i}(\cdot)\right)&=\frac{1}{2} \mathbb{E}\left[\int_0^T\left(\left\|x_i(t)-\Gamma_1(t) x^{(N)}(t)-\eta_1(t)\right\|_Q^2+\left\|u_i(t)\right\|_R^2+\sum_{j=1}^N\left\|z_{ij}(t)\right\|_{S_j}^2\right) d t\right. \\
	&\qquad\quad +\left\|x_i(0)-\Gamma_0 x^{(N)}(0)-\eta_0\right\|_G^2\Bigg],
\end{aligned}
\end{equation}
where $\left\|u_i(t)\right\|_R^2\equiv\left\langle R(t)u_i(t),u_i(t)\right\rangle$, etc., and $Q(\cdot),R(\cdot),S_j(\cdot),\Gamma_1(\cdot),\eta_1(\cdot)$ are deterministic functions with compatible dimensions. Let $\mathcal{F}_t^i=\sigma\left(W_i(s), 0 \leqslant s \leqslant t \right)$. Define the centralized control set of agent $\mathcal{A}_i$ as 
$$
\mathcal{U}^c_i= \bigg\{u_i(\cdot) \mid u_i(t) \text { is adapted to } \mathcal{F}_t \text { and } 
\mathbb{E} \int_0^T\left|u_i(t)\right|^2 d t<\infty\bigg\},
$$
and the decentralized control set of agent $\mathcal{A}_i$ as 
$$
\mathcal{U}^d_i= \bigg\{u_i(\cdot) \mid u_i(t) \text { is adapted to } \mathcal{F}^i_t \text { and }
\mathbb{E} \int_0^T\left|u_i(t)\right|^2 d t<\infty\bigg\}.
$$

In this section, we mainly study the two problems:
\begin{Problem}\label{problem centralized}
	Seek a Nash equilibrium strategy $u^*(\cdot)=(u^*_1(\cdot),\ldots,u^*_N(\cdot))$, $u^*_i(\cdot) \in \mathcal{U}^c_i$ for the system (\ref{state})-(\ref{cost}), i.e., $\mathcal{J}_i(u^*_i(\cdot);u^*_{-i}(\cdot)) = \inf\limits_{u_i(\cdot) \in \mathcal{U}^c_i} \mathcal{J}_i(u_i(\cdot);u^*_{-i}(\cdot))$, $i=1,\ldots,N$.
\end{Problem}

\begin{Problem}\label{problem decentralized}
	For $\epsilon>0$, seek an $\epsilon$-Nash equilibrium strategy $u^*(\cdot)=(u^*_1(\cdot),\ldots,u^*_N(\cdot))$, $u^*_i(\cdot) \in \mathcal{U}^d_i$ for the system (\ref{state})-(\ref{cost}), i.e., $\mathcal{J}_i(u^*_i(\cdot);u^*_{-i}(\cdot)) \leq  \inf\limits_{u_i(\cdot) \in \mathcal{U}^d_i} \mathcal{J}_i(u_i(\cdot);u^*_{-i}(\cdot))+\epsilon$, $i=1,\ldots,N$.
\end{Problem}

Next, we introduce the following assumptions.
\begin{Assumption}\label{A1}
	The coefficients satisfy the following conditions:
	
	(i) $A(\cdot), \Gamma_1(\cdot) \in L^\infty\left(0, T ; \mathbb{R}^{n \times n}\right)$, and $B(\cdot) \in L^\infty(0, T$; $\left.\mathbb{R}^{n \times k}\right) ;$
	
	(ii) $Q(\cdot),S_j(\cdot) \in L^\infty\left(0, T ; \mathbb{S}^n\right)$, $R(\cdot) \in L^\infty\left(0, T ; \mathbb{S}^k\right)$, and $R(\cdot)>0$, $Q(\cdot) \geq 0$, $S_j(\cdot) \geq 0$;
	
	(iii) $\eta_1(\cdot) \in L^2\left(0, T ; \mathbb{R}^n\right)$;
	
	(iv) $\Gamma_0 \in \mathbb{R}^{n \times n}$, $\eta_0 \in \mathbb{R}^n$, $G \in \mathbb{S}^n$ are bounded and $G \geq 0$.
\end{Assumption} 

\begin{Assumption}\label{A2}
	The terminal conditions $\left\{\xi_i \in L_{\mathcal{F}_T}^2\left(\Omega ; \mathbb{R}^n\right), i=1,2, \cdots, N\right\}$ are identically distributed and mutually independent. There exists a constant $c_0$ (independent of $N$) such that $\max\limits_{1 \leq i \leq N} \mathbb{E}\left|\xi_i\right|^2<c_0$.
\end{Assumption}

\section{Design of the decentralized strategies}

\begin{Lemma}
	Let (\ref{A1}) and (\ref{A2}) hold, then $\mathcal{J}_i(u(\cdot);u^*_{-i}(\cdot))$, $i=1,\ldots,N$ is a strictly convex functional.
\end{Lemma}
\begin{proof}
	The proof is similar to \cite{Du-Huang-Wu-2018}, and we will not repeat it.
\end{proof}
\begin{Remark}
	If $J_i(u(\cdot);u^*_{-i}(\cdot))$, $i=1,\ldots,N$ is uniformly convex, then Problem (\ref{problem centralized}) exists a unique Nash equilibrium strategy $u^*(\cdot)=(u^*_1(\cdot),\ldots,u^*_N(\cdot))$.
\end{Remark}

We first obtain the necessary and sufficient conditions for the existence of centralized optimal control of Problem (\ref{problem centralized}). For notational simplicity, the time variable $t$ is often omitted.

\begin{Theorem}
	Assume (\ref{A1}) and (\ref{A2}) hold. Then Problem (\ref{problem centralized}) has a Nash equilibrium strategy $u^*(\cdot)=(u^*_1(\cdot),\ldots,u^*_N(\cdot))$, $u^*_i(\cdot) \in \mathcal{U}^c_i$ if and only if the following Hamiltonian system admits a set of solutions $(x^*_i(\cdot),z^*_{ij}(\cdot),p_i(\cdot),i,j=1,\ldots,N)$: 
	\begin{equation}\label{optimality Hamiltonian system}
		\left\{\begin{aligned}
			d x_i^*(t)&=\left[A(t) x_i^*(t)+B(t) u_i^*(t)\right] d t+\sum_{j=1}^N z_{ij}^*(t) d W_j(t), \\
			d p_i(t)&=-\left[A(t)^\top p_i(t)+\left(I_n-\frac{\Gamma_1(t)}{N}\right)^\top Q(t)\left(x_i^*(t)-\Gamma_1(t) x^{*(N)}(t)-\eta_1(t)\right)\right] d t\\
            &\quad -\sum_{j=1}^N S_j(t) z_{ij}^*(t) d W_j(t), \\
			x^*_i(T)&=\xi_i, \quad p_i(0)=-\left(I_n-\frac{\Gamma_0}{N}\right)^{\top} G\left(x_i^*(0)-\Gamma_0 x^{*(N)}(0)-\eta_0\right),
		\end{aligned}\right.
	\end{equation}
	and the centralized strategy $u^*_i(\cdot)$ satisfies the stationary condition:
	\begin{equation}\label{optimality conditions}
		u^*_i(t)=-R^{-1}(t)B(t)^\top p_i(t),\quad t\in[0,T].
	\end{equation}
	
\end{Theorem}

\begin{proof}
	Suppose $u^*(\cdot)=(u^*_1(\cdot),\ldots,u^*_N(\cdot))$ is a Nash equilibrium strategy of Problem (\ref{problem centralized}) and $(x^*_i(\cdot),z^*_{ij}(\cdot),i,j=1,\ldots,N)$ are the
	corresponding optimal trajectories. For any $u_i(\cdot) \in \mathcal{U}^c_i$ and $\forall$ $\varepsilon>0$, we denote 
	$$
	u_i^\varepsilon(\cdot)=u^*_i(\cdot)+\varepsilon v_i(\cdot) \in \mathcal{U}^c_i,
	$$
	where $v_i(\cdot)=u_i(\cdot)-u^*_i(\cdot)$.
	
	Let $(x^\varepsilon_i(\cdot),z^\varepsilon_{ij}(\cdot),i,j=1,\ldots,N)$ be the solution of the following perturbed state equation
	$$
	\left\{\begin{aligned}
		dx_i^\varepsilon=&\left(A x_i^\varepsilon+B u_i^\varepsilon\right) d t+\sum_{j=1}^Nz_{ij}^\varepsilon d W_j, \\
		x_i^\varepsilon(T)&=\xi_i.
	\end{aligned}\right.
	$$
	Let $x_i(\cdot)=\frac{x_i^\varepsilon(\cdot)-x_i^*(\cdot)}{\varepsilon},  z_{ij}(\cdot)=\frac{z_{ij}^\varepsilon(\cdot)-z_{ij}^*(\cdot)}{\varepsilon}$. It can be verified that $(x_i(\cdot),z_{ij}(\cdot),i,j=1,\ldots,N)$ satisfies
	$$
	\left\{\begin{aligned}
		d x_i=&\left(A x_i+B v_i\right) d t+\sum_{j=1}^{N}z_{ij} d W_j, \\
		x_i(T)&=0.
	\end{aligned}\right.
	$$
	Applying It\^o’s formula to $\left\langle x_i(\cdot), p_i(\cdot)\right\rangle$, we derive
	\begin{equation}\label{Ito maximum}
	\begin{aligned}
		&\mathbb{E}\left[0-\left\langle x_i(0),\left(I_n-\frac{\Gamma_0}{N}\right)^{\top} G\left(x^*_i(0)-\Gamma_0 x^{*(N)}(0)-\eta_0\right)\right\rangle\right]\\
        &=\mathbb{E} \int_0^T\left\langle B v_i, p_i\right\rangle
		-\left\langle x_i,\left(I_n-\frac{\Gamma_1}{N}\right)^T Q\left(x^*_i-\Gamma_1 x^{*(N)}-\eta_1\right)\right\rangle-\sum_{j=1}^{N}\left\langle z_{ij}, S_j z^*_{ij}\right\rangle d t.
	\end{aligned}
	\end{equation}
	Then
	$$
	\begin{aligned}
		 \mathcal{J}_i\left(u_i^\varepsilon(\cdot); u^*_{-i}(\cdot)\right)-\mathcal{J}_i\left(u^*_i (\cdot); u^*_{-i}(\cdot)\right)=\frac{\varepsilon^2}{2} X_1+\varepsilon X_2,
	\end{aligned}
	$$
	where 
	$$
	\begin{aligned}
		X_1&= \mathbb{E}\left[\int_0^T \left[\left(\left(I_n-\frac{\Gamma_1}{N}\right) x_i\right)^\top Q\left(\left(I_n-\frac{\Gamma_1}{N}\right) x_i\right)
        +v_i^\top R v_i+\sum_{j=1}^{N}z_{ij}^\top S_j  z_{ij}\right] d t\right. \\
	    &\quad \left.\quad+\left((I_n-\frac{\Gamma_0}{N})  x_i(0)\right)^\top G\left(\left(I_n-\frac{\Gamma_0}{N}\right) x_i(0)\right)\right],
    \end{aligned}
	$$
	\begin{equation}\label{X_2}
	\begin{aligned}
		X_2&= \mathbb{E}{\left[\int_0^T\left[\left(x^*_i-\Gamma_1 x^{*(N)}-\eta_1\right)^\top Q\left(\left(I_n-\frac{\Gamma_1}{N}\right) x_i\right)+u_i^{*\top} R v_i+\sum_{j=1}^{N}z_{ij}^{*\top} H z_{ij} d t\right]\right.} \\
		&\quad \left.\quad+\left(x_i^*(0)-\Gamma_0 x^{*(N)}(0)-\eta_0\right)^\top G\left(\left(I_n-\frac{\Gamma_0}{N}\right) x_i(0)\right)\right].	
	\end{aligned}
	\end{equation}
	Due to the optimality of $u^*_i(\cdot)$, we have $\mathcal{J}_i\left(u_i^\varepsilon(\cdot); u^*_{-i}(\cdot)\right)-\mathcal{J}_i\left(u^*_i(\cdot); u^*_{-i}(\cdot)\right) \geqslant 0$. Noticing $X_1 \geqslant 0$ and the arbitrariness of $\varepsilon$, we have $X_2=0$. Then, simplifying (\ref{X_2}) with (\ref{Ito maximum}), we have
	$$
	X_2=\mathbb{E} \int_0^T\left\langle B^\top p_i+R u^*_i, v_i\right\rangle d t.
	$$
	Due to the arbitrariness of $v_i(\cdot)$, we obtain the optimal conditions (\ref{optimality conditions}).
\end{proof}

Note that optimality conditions (\ref{optimality conditions}) are open-loop centralized strategies. The next step is to obtain proper form for the centralized feedback representation of optimality conditions. Since state equation is backward, we divide the decoupling procedure into two steps, inspired by \cite{Lim-Zhou-2001}.

\begin{Proposition}
	Let Assumption (\ref{A1}), (\ref{A2}) hold. Let $(x^*_i(\cdot),z^*_{ij}(\cdot),p_i(\cdot),i,j=1,\ldots,N)$ be the solution of FBSDE (\ref{optimality Hamiltonian system}). Then, we have the following relations:
	\begin{equation}
		\left\{\begin{aligned}
			x^*_i(t)&=\Sigma(t) p_i(t)+K(t) p^{(N)}(t)+\varphi_i(t),\\
			z^*_{ij}(t)&=\left(I_n+\Sigma(t) S_j(t)\right)^{-1} \beta_{i j}(t)-\frac{K_1(t)}{N} \sum_{i=1}^{N} \beta_{i j}(t),
		\end{aligned}\right.
	\end{equation}
	where
	$$
	\begin{aligned}
	    K_1(t)=\left(I_n+\Sigma(t) S_j(t)\right)^{-1} K(t)S_j(t) \left(I_n+\Sigma(t) S_j(t)+K(t) S_j(t)\right)^{-1},    
    \end{aligned}
	$$
	and $\Sigma(\cdot),K(\cdot),\varphi_i(\cdot),\beta_{i j}(\cdot)$ are solutions of the following equations, respectively:
	\begin{equation}\label{Sigma equation}
		\left\{\begin{aligned}
			&\dot{\Sigma}-A \Sigma-\Sigma A^\top-\Sigma\left(I_n-\frac{\Gamma_1}{N}\right)^\top Q \Sigma+B R^{-1} B^\top=0, \\
			&\Sigma(T)=0,
		\end{aligned}\right.
	\end{equation}
	\begin{equation}\label{K equation}
		\left\{\begin{aligned}
			& \dot{K}-A K-K A^{\top}-\Sigma\left(I_n-\frac{\Gamma_1}{N}\right) ^\top Q K-K\left(I_n-\frac{\Gamma_1}{N}\right)^{\top} Q(\Sigma+K)  \\
			& +(K+\Sigma)\left(I_n-\frac{\Gamma_1}{N}\right)^\top Q \Gamma_1(\Sigma+K)=0, \\
			& K(T)=0,
		\end{aligned}\right.
	\end{equation}
	\begin{equation}\label{varphi equation}
		\left\{\begin{aligned}
			 d \varphi_i&=\left[A \varphi_i+\Sigma\left(I_n-\frac{\Gamma_1}{N}\right)^{\top} Q \varphi_i-(\Sigma+K)\left(I_n-\frac{\Gamma_1}{N}\right)^\top Q \Gamma_1 \varphi^{(N)}\right. \\
			&\quad \left.+K\left(I_n-\frac{\Gamma_1}{N}\right)^\top Q \varphi^{(N)}-(\Sigma+K)\left(I_n-\frac{\Gamma_1}{N}\right)^\top Q \eta_1\right]dt+\sum_{j=1}^N \beta_{i j} d W_j,\\
			\varphi_i(T)&=\xi_i.
		\end{aligned}\right.
	\end{equation}
\end{Proposition}

\begin{proof}
	Noting the terminal condition and structure of (\ref{optimality Hamiltonian system}), for each $i = 1,\ldots, N$, we suppose
	\begin{equation}\label{decouple form}
		x^*_i(\cdot)=\Sigma(\cdot) p_i(\cdot)+K(\cdot) p^{(N)}(\cdot)+\varphi_i(\cdot),
	\end{equation}
	with $\Sigma(T) = 0, K(T) = 0$ for two deterministic differentiable functions $\Sigma(\cdot), K(\cdot)$, and for an $\mathcal{F}_t^i$-adapted process $\varphi_i(\cdot)$ satisfying a BSDE:
	\begin{equation}
		\left\{\begin{aligned}
			d \varphi_i(t)&=\alpha_i(t) d t+ \sum_{j=1}^N \beta_{i j}(t) d W_j(t),\\
			\varphi_i(T)&=\xi_i,
		\end{aligned}\right.
	\end{equation}
	with process $\alpha_i(\cdot)$ to be determined. And then, we have
	\begin{equation}\label{decouple form x^N}
        x^{*(N)}(\cdot)=\Sigma(\cdot) p^{(N)}(\cdot)+K(\cdot) p^{(N)}(\cdot)+\varphi^{(N)}(\cdot).
	\end{equation}
	Applying It\^o’s formula to (\ref{decouple form}), we have
	$$
	\begin{aligned}
		dx^*_i&= \dot{\Sigma} p_i d t-\Sigma\left\{\left[A^\top p_i+\left(I_n-\frac{\Gamma_1}{N}\right)^\top Q\left(x^*_i-\Gamma_1 x^{*(N)}-\eta_1\right)\right] d t+\sum_{j=1}^N S_j z_{i j} d W_j\right\}  \\
		&\quad + \dot{K} p^{(N)} d t-K\left\{\left[A^\top p^{(N)}+\left(I_n-\frac{\Gamma_1}{N}\right)^\top Q\left(x^{*(N)}-\Gamma_1 x^{*(N)}-\eta_1\right)\right] d t\right. \\
		&\quad \left.+\sum_{j=1}^N S_j \frac{1}{N} \sum_{i=1}^N z^*_{i j} d W_j\right\} +\alpha_i d t+\sum_{j=1}^N \beta_{i j} d W_j  \\
		& =\left(A x_i^*-B R^{-1} B^{\top}p_i\right) d t+\sum_{j=1}^N z^*_{i j} d W_j.
	\end{aligned}
	$$
	By comparing the coefficients of the diffusion terms, we obtain
	$$
	-\Sigma \sum_{j=1}^N S_j z^*_{i j} -\frac{K}{N}  \sum_{j=1}^N S_j \sum_{i=1}^N z^*_{i j} +\sum_{j=1}^N \beta_{i j}  -\sum_{j=1}^N z_{ij} =0.
	$$
	Then we can solve for $z^*_{ij}(\cdot)$ explicitly:
	$$
	z^*_{ij}=\left(I_n+\Sigma S_j\right)^{-1} \beta_{i j}-\left(I_n+\Sigma S_j\right)^{-1} \frac{K}{N}S_j \left(I_n+\Sigma S_j+K S_j\right)^{-1} \sum_{i=1}^{N} \beta_{i j}.
	$$
	Then by comparing the coefficients of the drift terms, we obtain 
	$$
	\begin{aligned}
	    &\dot{\Sigma} p_i-\Sigma\left[A^\top p_i+\left(I_n-\frac{\Gamma_1}{N}\right)^\top Q\left(x^*_i-\Gamma_1 x^{*(N)}-\eta_1\right)\right]+ \dot{K} p^{(N)} \\
	    &-K\left[A^\top p^{(N)}+\left(I_n-\frac{\Gamma_1}{N}\right)^\top Q\left(x^{*(N)}-\Gamma_1 x^{*(N)}-\eta_1\right)\right]+\alpha_i\\
	    & =A x_i^*-B R^{-1} B^\top p_i.
    \end{aligned}
	$$
	Combining (\ref{decouple form}) and (\ref{decouple form x^N}), we can obtain the equation (\ref{Sigma equation}) of $\Sigma(\cdot)$  from the coefficients of $p_i(\cdot)$, the equation (\ref{K equation}) of $K(\cdot)$  from the coefficients of $p^{(N)}(\cdot)$ and the equation (\ref{varphi equation}) of $\varphi_i(\cdot)$ from the non-homogeneous term. Then, we completed the proof.
\end{proof}

\begin{Proposition}\label{decouple proposition 2}
	Let Assumption (\ref{A1}), (\ref{A2}) hold. Let $(x^*_i(\cdot),z^*_{ij}(\cdot),p_i(\cdot),i,j=1,\ldots,N)$ be the solution of FBSDE (\ref{optimality Hamiltonian system}). Then, we have the following relations:
	\begin{equation}\label{decouple form 2}
		p_i(\cdot)=\Pi(\cdot) x^*_i(\cdot)+M(\cdot) x^{*(N)}(\cdot)+\zeta_i(\cdot),
	\end{equation}
	where $\Pi(\cdot),M(\cdot),\zeta_i(\cdot)$ are solutions of the following equations, respectively:
	\begin{equation}\label{Pi equation}
		\left\{\begin{aligned}
			&\dot{\Pi}+\Pi A+ A^\top \Pi-\Pi B R^{-1} B^\top \Pi+\left(I_n-\frac{\Gamma_1}{N}\right)^\top Q=0, \\
			&\Pi(0)=-\left(I_n-\frac{\Gamma_0}{N}\right)^{\top} G,
		\end{aligned}\right.
	\end{equation}
	\begin{equation}\label{M equation}
		\left\{\begin{aligned}
			&\dot{M}+M A+A^\top M-\Pi B R^{-1} B^\top M-M B R^{-1} B^\top(\Pi+M)-\left(I_n-\frac{\Gamma_1}{N}\right)^\top Q \Gamma_1=0, \\
			&M(0)=\left(I_n-\frac{\Gamma_0}{N}\right)^{\top} G \Gamma_0,
		\end{aligned}\right.
	\end{equation}
	\begin{equation}\label{zeta equation}
		\left\{\begin{aligned}
			d \zeta_i=&\left[\left(\Pi B R^{-1} B^\top-A^\top\right) \zeta_i+M B R^{-1} B^\top \zeta^{*(N)}+\left(I_n-\frac{\Gamma_1}{N}\right)^\top Q \eta_1 \right]dt \\
			& -\sum_{j=1}^N \left(S_j+\Pi\right) \left(I_n+\Sigma S_j\right)^{-1} \beta_{i j}dW_j -\sum_{j=1}^N\frac{M}{N} \left(I_n+\Sigma S_j\right)^{-1}\sum_{i=1}^N \beta_{i j} d W_j  \\
			& +\sum_{j=1}^N \left(S_j + \Sigma-M \right)\frac{K_1}{N} \sum_{i=1}^N \beta_{i j} d W_j, \\
			\zeta_i(0)=&\left(I_n-\frac{\Gamma_0}{N}\right)^{\top} G \eta_0.
		\end{aligned}\right.
	\end{equation}
\end{Proposition}
\begin{proof}
	Noting the initial condition and structure of (\ref{optimality Hamiltonian system}), for each $i = 1,\ldots, N$, we suppose
	\begin{equation}\label{decouple form 2}
		p_i(\cdot)=\Pi(\cdot) x^*_i(\cdot)+M(\cdot) x^{*(N)}(\cdot)+\zeta_i(\cdot),
	\end{equation}
	with $\Pi(0) = (I_n-\frac{\Gamma_0}{N})^\top G, M(0) = (I_n-\frac{\Gamma_0}{N})^\top G \Gamma_0$ for two deterministic differentiable functions $\Pi(\cdot), M(\cdot)$, and for an $\mathcal{F}_t^i$-adapted process $\zeta_i(\cdot)$ satisfying an SDE:
	$$
	\left\{\begin{aligned}
		d \zeta_i(t)&=\chi_i(t) d t+ \sum_{j=1}^N \gamma_{i j}(t) d W_{j}, \\
		\zeta_i(0)&=(I_n-\frac{\Gamma_0}{N})^\top G \eta_0,
	\end{aligned}\right.
	$$
	with processes $\chi_i(\cdot), \gamma_{ij}(\cdot)$ to be determined. Applying It\^o’s formula to (\ref{decouple form 2}), we have
	$$
	\begin{aligned}
		dp_i&=\left\{\dot{\Pi} x^*_i+\Pi A x^*_i-\Pi B R^{-1} B^{\top}\left(\Pi x^*_i+M x^{*(N)}+\zeta^*_i\right)\right\}dt+ \Pi \sum_{j=1}^{N} z^*_{ij} dW_j\\
        &\quad +\left\{\dot{M} x^{*(N)}+M A x^{*(N)}-M B R^{-1} B^\top \left[\left(\Pi+M\right) x^{*(N)} +\zeta^{*(N)}\right] +\chi_i \right\}dt\\
		&\quad +\frac{M}{N}\sum_{j=1}^{N}\sum_{i=1}^{N} z^*_{ij} dW_j+\sum_{j=1}^{N} \gamma_{ij} dW_j \\
		&=-\left[A^{\top}\left(\Pi x^*_i+M x^{*(N)}+\zeta^*_i\right)+\left(I_n-\frac{\Gamma_1}{N}\right)^\top Q\left(x^*_i-\Gamma_1 x^{*(N)}-\eta_1\right)\right]dt -\sum_{j=1}^{N} S_j z^*_{ij} dW_j.
	\end{aligned}
	$$
	By comparing the coefficients of the diffusion terms, we obtain
	$$
	-\sum_{j=1}^N S_j z^*_{i j}=\sum_{j=1}^N \gamma_{i j} + \Pi \sum_{j=1}^N z^*_{i j} +\frac{M}{N} \sum_{j=1}^N \sum_{j=1}^N z^*_{i j}.
	$$
	Then we can solve for $\gamma_{ij}(\cdot)$ explicitly:
	$$
	\gamma_{i j}=-S_j z^*_{i j}-\Pi z^*_{i j}-\frac{M}{N} \sum_{i=1}^N z^*_{i j}.
	$$
	Then by comparing the coefficients of the drift terms, we obtain 
	$$
	\begin{aligned}
		&\dot{\Pi} x^*_i+\Pi A x^*_i-\Pi B R^{-1} B^\top\left(\Pi x^*_i+M x^{*(N)}+\zeta_i\right)\\
		&+\dot{M} x^{*(N)}+M A x^{*(N)}-M B R^{-1} B^\top \left[\left(\Pi+M\right) x^{*(N)} +\zeta^{(N)}\right] +\chi_i \\
		&=-\left[A^{\top}\left(\Pi x^*_i+M x^{*(N)}+\zeta_i\right)+\left(I_n-\frac{\Gamma_1}{N}\right)^\top Q\left(x^*_i-\Gamma_1 x^{*(N)}-\eta_1\right)\right].
	\end{aligned}
	$$
    Then, we can obtain the equation (\ref{Pi equation}) of $\Pi(\cdot)$  from the coefficients of $x^*_i(\cdot)$, the equation (\ref{M equation}) of $M(\cdot)$  from the coefficients of $x^{*(N)}(\cdot)$ and
	the equation (\ref{zeta equation}) of $\zeta_i(\cdot)$  from the non-homogeneous term. Then, we completed the proof.
\end{proof}

\begin{Theorem}
	Let Assumption (\ref{A1}), (\ref{A2}) hold. Then Riccati equations (\ref{Sigma equation}), (\ref{K equation}), (\ref{Pi equation}), (\ref{M equation}), BSDE (\ref{varphi equation}), SDE (\ref{zeta equation}) admit unique solutions, respectively. In addition, the centralized optimal strategy of agent $\mathcal{A}_i$ has a feedback form as follows:
	\begin{equation}\label{centralized strategy}
		u^*_i(t)=-R^{-1}(t) B^\top (t)\left(\Pi(t) x^*_i(t)+M(t) x^{*(N)}(t)+\zeta_i(t)\right),\quad t\in[0,T].
	\end{equation}
\end{Theorem}
\begin{proof}
	By referring to monograph \cite{Reid-1972}, Riccati equations (\ref{Sigma equation}), (\ref{K equation}), (\ref{Pi equation}), (\ref{M equation}) have unique solutions, respectively. The existence and uniqueness of the solution to $\varphi_i(\cdot)$ and $\zeta_i(\cdot)$ can be derived by
	classical linear SDE and BSDE theory.
	Applying Proposition (\ref{decouple proposition 2}), we can obtain (\ref{centralized strategy}).
\end{proof}

Here, we have got the centralized optimal strategy of agent $\mathcal{A}_i$. Next, to overcome the difficulty posed by the curse of dimensionality induced by state-average term $x^{*(N)}$ in numerical calculation, we let $N$ tend to infinity to obtain a decentralized optimal strategy of agent $\mathcal{A}_i$. 

First of all, we need to obtain the limiting version of Riccati equations (\ref{Sigma equation}), (\ref{K equation}), (\ref{Pi equation}), (\ref{M equation}):
\begin{equation}\label{Sigma limit equation}
	\left\{\begin{aligned}
		&\dot{\bar{\Sigma}}-A \bar{\Sigma}-\bar{\Sigma} A^\top-\bar{\Sigma} Q \bar{\Sigma}+B R^{-1} B^\top=0, \\
		&\bar{\Sigma}(T)=0,
	\end{aligned}\right.
\end{equation}
\begin{equation}\label{K limit equation}
	\left\{\begin{aligned}
		& \dot{\bar{K}}-A \bar{K}-\bar{K} A^\top-\bar{\Sigma} Q \bar{K}-\bar{K} Q(\bar{\Sigma}+\bar{K}) +(\bar{K}+\bar{\Sigma}) Q \Gamma_1(\bar{\Sigma}+\bar{K})=0, \\
		& \bar{K}(T)=0,
	\end{aligned}\right.
\end{equation}
	\begin{equation}\label{Pi limit equation}
	\left\{\begin{aligned}
		&\dot{\bar{\Pi}}+\bar{\Pi} A+ A^\top \bar{\Pi}-\bar{\Pi} B R^{-1} B^\top \bar{\Pi}+ Q=0, \\
		&\bar{\Pi}(0)=- G,
	\end{aligned}\right.
\end{equation}
\begin{equation}\label{M limit equation}
	\left\{\begin{aligned}
		&\dot{\bar{M}}+\bar{M} A+A^\top \bar{M}-\bar{\Pi} B R^{-1} B^\top \bar{M}-\bar{M} B R^{-1} B^\top(\bar{\Pi}+\bar{M})- Q \Gamma_1 =0\\
		&\bar{M}(0)= G \Gamma_0.
	\end{aligned}\right.
\end{equation}
\begin{Remark}
	By referring to \cite{Reid-1972}, Riccati equations (\ref{Sigma limit equation}), (\ref{K limit equation}), (\ref{Pi limit equation}), (\ref{M limit equation}) have unique solutions, respectively. And we can use the
	continuous dependence of solutions on the parameter referred to the Theorem 3.5 of \cite{Khalil-2002} to verifies the limiting functions of $\Sigma(\cdot)$, $K(\cdot)$, $\Pi(\cdot)$, $M(\cdot)$ are $\bar{\Sigma}(\cdot)$, $\bar{K}(\cdot)$, $\bar{\Pi}(\cdot)$, $\bar{M}(\cdot)$, respectively. And applying Theorem 4 of \cite{Huang-Zhou-2020}, we have $\sup\limits_{0\leqslant t \leqslant T}|\Sigma(t)-\bar{\Sigma}(t)|=O\left(\frac{1}{N}\right)$, $\sup\limits_{0\leqslant t \leqslant T}|K(t)-\bar{K}(t)|=O\left(\frac{1}{N}\right)$, $\sup\limits_{0\leqslant t \leqslant T}|\Pi(t)-\bar{\Pi}(t)|=O\left(\frac{1}{N}\right)$, $\sup\limits_{0\leqslant t \leqslant T}|M(t)-\bar{M}(t)|=O\left(\frac{1}{N}\right)$.
\end{Remark}
	
    Next, by summing up $N$ equations of (\ref{varphi equation}) and (\ref{zeta equation}) and dividing them by $N$, we get
	\begin{equation}\label{varphi N equation}
	\left\{\begin{aligned}
		d \varphi^{(N)}&= \left\{\left[A +\left(\Sigma+K\right)\left(I_n-\frac{\Gamma_1}{N}\right)^\top Q \left(I_n-\Gamma_1\right)\right] \varphi^{(N)}\right.\\
		&\quad \left.-(\Sigma+K)\left(I_n-\frac{\Gamma_1}{N}\right)^\top Q \eta_1\right\}dt+\frac{1}{N}\sum_{j=1}^N \sum_{i=1}^N \beta^*_{i j} d W_j,\\
		\varphi^{(N)}(T)&=\frac{1}{N}\sum_{i=1}^N \xi_i,
	\end{aligned}\right.
\end{equation}
\begin{equation}\label{zeta N equation}
	\left\{\begin{aligned}
		d \zeta^{(N)}=&\left\{\left[\left(\Pi+M\right) B R^{-1} B^\top-A^\top\right] \zeta^{(N)}+\left(I_n-\frac{\Gamma_1}{N}\right)^\top Q \eta_1 \right\}dt \\
		& -\sum_{j=1}^N \left(S_j+\Pi+M\right) \left(I_n+\Sigma S_j+K S_j\right)^{-1}\frac{1}{N}\sum_{i=1}^N \beta_{i j}dW_j,\\
		\zeta^{(N)}(0)=&\left(I_n-\frac{\Gamma_0}{N}\right)^\top G \eta_0.
	\end{aligned}\right.
\end{equation}
When $N$ tends to infinity, by (\ref{A1}) and the strong Law of Large Numbers, the limit of $\frac{1}{N}\sum_{i=1}^N\xi_i$ exists and 
$$
\lim _{N \rightarrow \infty} \frac{1}{N} \sum_{i=1}^N \xi_i=\mathbb{E}\xi .
$$
Then, we can get the equations of the limiting processes $\bar{\varphi}(\cdot)$ and $\bar{\zeta}(\cdot)$ of $\varphi^{(N)}(\cdot)$ and $\zeta^{(N)}(\cdot)$:
\begin{equation}\label{varphi N limit equation}
	\left\{\begin{aligned}
		d \bar{\varphi}&= \Big\{\left[A +\left(\bar{\Sigma}+\bar{K}\right) Q \left(I_n-\Gamma_1\right)\right] \bar{\varphi}-(\bar{\Sigma}+\bar{K}) Q \eta_1\Big\}dt,\\
		\bar{\varphi}(T)&=\mathbb{E}\xi,
	\end{aligned}\right.
\end{equation}
\begin{equation}\label{zeta N limit equation}
	\left\{\begin{aligned}
		d \bar{\zeta}&= \left\{\left[\left(\bar{\Pi}+\bar{M}\right) B R^{-1} B^{\top}-A^\top\right] \bar{\zeta}+ Q \eta_1 \right\}dt, \\
		\bar{\zeta}(0)&= G \eta_0.
	\end{aligned}\right.
\end{equation}
\begin{Remark}
	Using the classical theory of SDEs and BSDEs, we have the following estimations: $\mathbb{E}\int_0^T|\varphi^{(N)}-\bar{\varphi}|^2dt=O\left(\frac{1}{N}\right)$, $\frac{1}{N}\sum_{j=1}^N \mathbb{E}\int_0^T|\sum_{i=1}^N \beta_{i j}|^2 dt=O\left(\frac{1}{N}\right)$, $\mathbb{E}\int_0^T|\varphi^{(N)}-\bar{\varphi}|^2 dt=O\left(\frac{1}{N}\right)$. 
\end{Remark}
Using a similar method as above, we can obtain the equation of $x^{*(N)}(\cdot)$:
\begin{equation}\label{x N system}
	\left\{\begin{aligned}
		d x^{*(N)}=&\left\{\left[A-BR^{-1}B^\top\left(\Pi+M\right)\right] x^{*(N)}-BR^{-1}B^\top \zeta^{(N)}\right\} d t+\frac{1}{N}\sum_{j=1}^N \sum_{i=1}^N z_{ij}^* d W_j, \\
		x^{*(N)}(T)&=\sum_{i=1}^N \xi_i.
	\end{aligned}\right.
\end{equation}
Then, letting $N$ tends to infinity, we can obtain the equations of the limiting processes $x_0(\cdot)$ of $x^{*(N)}(\cdot)$: 
\begin{equation}\label{x N limit system}
	\left\{\begin{aligned}
		dx_0=&\left\{\left[A-BR^{-1}B^{\top}\left(\bar{\Pi}+\bar{M}\right)\right] x_0-BR^{-1}B^\top\bar{\zeta}\right\} d t,\\
		x_0(T)&=\mathbb{E}\xi.
	\end{aligned}\right.
\end{equation}

Next, we replace the state-average term by the limiting processes $\bar{\varphi}(\cdot)$ and $\bar{\zeta}(\cdot)$ in BSDE (\ref{varphi equation}) and SDE (\ref{zeta equation}), respectively. Therefore, we derive the following equations, which is decoupled, for $i=1,\cdots,N$:
\begin{equation}\label{varphi limit equation}
	\left\{\begin{aligned}
		 d \bar{\varphi}_i&=\Big\{\left(A+\bar{\Sigma} Q\right) \bar{\varphi}_i-\left[(\bar{\Sigma}+\bar{K}) Q \Gamma_1+\bar{K}Q\right] \bar{\varphi}-(\bar{\Sigma}+\bar{K}) Q \eta_1\Big\}dt+\sum_{j=1}^N \bar{\beta}_{i j} d W_j,\\
		\bar{\varphi}_i(T)&=\xi_i,
	\end{aligned}\right.
\end{equation}
	\begin{equation}\label{zeta limit equation}
	\left\{\begin{aligned}
		d \bar{\zeta}_i&=\left[\left(\bar{\Pi} B R^{-1} B^\top-A^{\top}\right) \bar{\zeta}_i+\bar{M} B R^{-1} B^\top \bar{\zeta}+ Q \eta_1 \right]dt \\
		&\quad -\sum_{j=1}^N \left(S_j+\bar{\Pi}\right) \left(I_n+\bar{\Sigma} S_j\right)^{-1} \bar{\beta}_{i j}dW_j, \\
		\zeta^*_i(0)&= G \eta_0.
	\end{aligned}\right.
\end{equation}
\begin{Remark}
	Similarly, using the classical theory of SDEs and BSDEs, we have the following estimations: $\mathbb{E}\int_0^T|x^{*(N)}-x_0|^2dt=O\left(\frac{1}{N}\right)$, $\frac{1}{N}\sum_{j=1}^N \mathbb{E}\int_0^T|\sum_{i=1}^N z^{*}_{i j}|^2 dt=O\left(\frac{1}{N}\right)$, 
	$\mathbb{E}\int_0^T|\varphi_i-\bar{\varphi}_i|^2 dt=O\left(\frac{1}{N}\right)$,  $\mathbb{E}\int_0^T\sum_{i=1}^N |\beta_{i j}-\bar{\beta}_{i j}|^2 dt=O\left(\frac{1}{N}\right)$,
    $\mathbb{E}\int_0^T|\zeta_i-\bar{\zeta}_i|^2 dt=O\left(\frac{1}{N}\right)$. 
\end{Remark}
Similarly, replacing $x^{*(N)}(\cdot)$ by its limiting process $x_0(\cdot)$ and replacing $\zeta_i(\cdot)$ by $\bar{\zeta}_i(\cdot)$, we have the decoupled state equation of $\bar{x}_i(\cdot)$:
\begin{equation}\label{x_i limit system}
	\left\{\begin{aligned}
		d\bar{x}_i&=\left\{\left(A-BR^{-1}B^\top \bar{\Pi}\right) \bar{x}_i-BR^{-1}B^\top \bar{M} x_0-BR^{-1}B^\top \bar{\zeta}_i\right\} d t+\sum_{j=1}^N \bar{z}_{ij} d W_j, \\
		\bar{x}_i(T)&=\xi_i.
	\end{aligned}\right.
\end{equation}
Using $\bar{x}_i(\cdot)$, $x_0(\cdot)$ and $\bar{\zeta}_i(\cdot)$ to replace $x^*_i(\cdot)$, $x^{*(N)}(\cdot)$ and $\zeta_i(\cdot)$ in the centralized strategy (\ref{centralized strategy}), respectively, we obtain the decentralized strategy for agent $\mathcal{A}_i$:
\begin{equation}\label{decentralized strategy}
	\bar{u}_i(t)=-R^{-1}(t) B(t)^\top\left(\bar{\Pi}(t) \bar{x}_i(t)+\bar{M}(t) x_0(t)+\bar{\zeta}_i(t)\right),\quad t\in[0,T].
\end{equation}
\begin{Remark}
	Similarly, we can also have the following estimations: 
	$\mathbb{E}\int_0^T|x^*_i-\bar{x}_i|^2 dt=O\left(\frac{1}{N}\right)$,  $\mathbb{E}\int_0^T\sum_{i=1}^N |z^{*}_{i j}-\bar{z}_{i j}|^2 dt=O\left(\frac{1}{N}\right)$. 
\end{Remark}
We will verify its $\epsilon$-asymptotic property in the next section.

\section{The asymptotic analysis}

In this section, we aim to prove that the decentralized strategies (\ref{decentralized strategy}) of agent $\mathcal{A}_i$, $i=1,\ldots,N$
constitute an approximated $\epsilon$-Nash equilibrium.
\begin{Theorem}
	Let Assumption (\ref{A1}), (\ref{A2}) hold. Then $(\bar{u}_1(\cdot),\ldots,\bar{u}_N(\cdot))$ given by (\ref{decentralized strategy}) is an $\epsilon$-Nash equilibrium of Problem (\ref{problem decentralized}), where $\epsilon=O(\frac{1}{\sqrt{N}})$, i.e.,
	$$
	\left| \mathcal{J}_i(\bar{u_i}(\cdot);\bar{u}_{-i}(\cdot))- \inf _{u_i(\cdot) \in \mathcal{U}_i^c} J_i(u_i(\cdot);\bar{u}_{-i}(\cdot))\right|=O\left(\frac{1}{\sqrt{N}}\right) .
	$$
\end{Theorem}
\begin{proof}
	First, by summing up $N$ equations of (\ref{x_i limit system}) and dividing them by $N$, we have
	\begin{equation}\label{bar x^N limit system}
		\left\{\begin{aligned}
			d\bar{x}^{(N)}&=\left\{\left(A-BR^{-1}B^\top \bar{\Pi}\right) \bar{x}^{(N)}-BR^{-1}B^\top \bar{M} x_0-BR^{-1}B^\top \bar{\zeta}_i\right\} d t+\frac{1}{N}\sum_{j=1}^N \sum_{i=1}^N \bar{z}_{ij} d W_j, \\
			\bar{x}^{(N)}(T)&=\frac{1}{N}\sum_{i=1}^N \xi_i.
		\end{aligned}\right.
	\end{equation}
	Using the classical theory of BSDEs, we have the following estimations: $\mathbb{E}\int_0^T|\bar{x}^{(N)}-x_0|^2dt=O\left(\frac{1}{N}\right)$, $\frac{1}{N}\sum_{j=1}^N \mathbb{E}\int_0^T|\sum_{i=1}^N \bar{z}_{i j}|^2 dt=O\left(\frac{1}{N}\right)$.
	
	For any $u_i(\cdot) \in \mathcal{U}_i^c$, let $\tilde{u}_i(\cdot)=u_i(\cdot)-\bar{u}_i(\cdot)$, $\tilde{x}_i(\cdot)=x_i(\cdot)-\bar{x}_i(\cdot)$, $\tilde{z}_{ij}(\cdot)=z_{ij}(\cdot)-\bar{z}_{ij}(\cdot)$
	where $(x_i(\cdot),z_{ij}(\cdot))$ denote the state processes corresponding to $u_i(\cdot)$. 
	Then, $(\tilde{x}_i(\cdot),\tilde{z}_{ij}(\cdot))$ satisfies
	$$
	\left\{\begin{aligned}
		d\tilde{x}_i&=\left[A\tilde{x}_i+B\tilde{u}_i\right]dt+\sum_{j=1}^N \tilde{z}_{ij}dW_j,\\
		\tilde{x}_i(T)&=0.
	\end{aligned}\right.
	$$
	Then, from (\ref{cost}), we have 
	\begin{equation}\label{asymptotic property}
		\mathcal{J}_i(u_i(\cdot);\bar{u}_{-i}(\cdot))-\mathcal{J}_i(\bar{u}_i(\cdot);\bar{u}_{-i}(\cdot))=\widetilde{\mathcal{J}}_i(\tilde{u}_i(\cdot);\bar{u}_{-i}(\cdot))+\mathcal{I}_i,
	\end{equation}
	where
	$$
	\begin{aligned}
		\widetilde{\mathcal{J}}_i(\tilde{u}_i(\cdot);\bar{u}_{-i}(\cdot))&=\frac{1}{2} \mathbb{E}\left[\int_0^T\left(\left\|\tilde{x}_i-\frac{\Gamma_1}{N}\tilde{x}_i\right\|_Q^2+\left\|\tilde{u}_i\right\|_R^2+\sum_{j=1}^N \left\|\tilde{z}_{i j}\right\|_{S_j}^2\right) d t+\left\|\tilde{x}_i(0)-\frac{\Gamma_0}{N} \tilde{x}_i(0)\right\|^2_G\right],
	\end{aligned}
	$$
	\begin{equation}\label{I_i equation}
	\begin{aligned}
		\mathcal{I}_i&= \mathbb{E}\left[ \int_0^T\left(\left(\bar{x}_i-\Gamma_1 \bar{x}_i^{(N)}-\eta_1\right)^\top Q\left(\tilde{x}_i-\frac{\Gamma_1}{N} \tilde{x}_i\right)+\bar{u}_i^\top R \tilde{u}_i+\sum_{j=1}^N \bar{z}_{i j}^\top S_j \tilde{z}_{i j}\right) d t \right.\\
		&\qquad +\left(\bar{x}_i(0)-\Gamma_0 \bar{x}_i^{(N)}(0)-\eta_0\right)^\top G\left(\tilde{x}_i(0)-\frac{\Gamma_0}{N}\tilde{x}_i(0)\right)\Bigg].
	\end{aligned}
    \end{equation}
	By Assumption (\ref{A1}), we have $\widetilde{\mathcal{J}}_i(\tilde{u}_i(\cdot);\bar{u}_{-i}(\cdot)) \geqslant 0$. 
	
	Let $\bar{p}_i(\cdot)=\bar{\Pi}(\cdot) \bar{x}_i(\cdot)+\bar{M}(\cdot) x_0(\cdot)+\bar{\zeta}_i(\cdot)$ and applying It\^o's formula to $\bar{p}_i(\cdot)$, we have
	$$
	\begin{aligned}
		d\bar{p}_i=& \left\{\left(-A^{\top} \bar{\Pi}-Q\right) \bar{x}_i+\left( -A^{\top} M+Q \Gamma_1\right) x_0-A^{\top} \bar{\zeta}_i+Q \eta_1\right\} d t \\
		& + \bar{\Pi} \sum_{j=1}^N \bar{z}_{i j} d W_j -\sum_{j=1}^N \left(S_j+\bar{\Pi}\right)\left(I_n+\bar{\Sigma} S_j\right)^{-1} \bar{\beta}_{i j} d W_j.
	\end{aligned}
	$$
	Using It\^o's formula to $\langle\tilde{x}_i^\top(\cdot),\bar{p}_i(\cdot)\rangle$, we derive
	$$
	\begin{aligned}
		&\mathbb{E}\left[\tilde{x}_i^\top (0)G\left(\bar{x}_i(0) +\Gamma_0 x_0(0) +\eta_0\right)\right]\\
		& =\mathbb{E}\Bigg[\Bigg( \int_0^T-\tilde{x}_i^{\top} Q \bar{x}_i+\tilde{x}_i^\top Q \Gamma_1 x_0+\tilde{x}_i^\top Q \eta_1+\tilde{u}_i^{\top} B^\top \left(\bar{\Pi} \bar{x}_i+\bar{M} x_0+\bar{\zeta}_i\right) \\
		&\qquad \left.\left. +\sum_{j=1}^N \tilde{z}_{i j}^\top \left(\bar{\Pi} \bar{z}_{i j}-\left(S_j+\bar{\Pi}\right)\left(I_n+\bar{\Sigma} S_j\right)^{-1} \bar{\beta}_{i j}\right)\right) d t \right].
	\end{aligned}
	$$
	Therefore, (\ref{I_i equation}) becomes
	$$
		\begin{aligned}
			\mathcal{I}_i&= \mathbb{E}\left[\left( \int_0^T\left(\bar{x}_i-\Gamma_1 \bar{x}_i^{(N)}-\eta_1\right)^\top Q\left(\tilde{x}_i-\frac{\Gamma_1}{N} \tilde{x}_i\right)+\bar{u}_i^\top R \tilde{u}_i+\sum_{j=1}^N \bar{z}_{i j}^\top S_j \tilde{z}_{i j}\right)dt \right. \\
			&\qquad \left.+\left(x_0(0)-\bar{x}_i^{(N)}(0)\right)^\top \Gamma_0^\top G\left(I_n-\frac{\Gamma_0}{N}\right)\tilde{x}_i(0)\right]\\
			&=\mathbb{E}\left[ \int_0^T\left(\left(\bar{x}_i-\Gamma_1 \bar{x}_i^{(N)}-\eta_1\right)^\top Q\left(\tilde{x}_i-\frac{\Gamma_1}{N} \tilde{x}_i\right)+\bar{u}_i^\top R \tilde{u}_i+\sum_{j=1}^N \bar{z}_{i j}^\top S_j \tilde{z}_{i j}\right) dt\right. \\
			&\qquad +\int_0^T\left(-\tilde{x}_i^\top Q \bar{x}_i+\tilde{x}_i^\top Q \Gamma_1 x_0+\tilde{x}_i^\top Q \eta_1+\tilde{u}_i^\top B^\top \left(\bar{\Pi} \bar{x}_i+\bar{M} x_0+\bar{\zeta}_i\right)\right. \\
			&\qquad \left.+\sum_{j=1}^N \tilde{z}_{i j}^\top \left(\bar{\Pi} \bar{z}_{i j}-\left(S_j+\bar{\Pi}\right)\left(I_n+\bar{\Sigma} S_j\right)^{-1} \bar{\beta}_{i j}\right)\right) d t  \\
			&\qquad -\left(\bar{x}_i(0)-\Gamma_0 x_0(0)-\eta_0\right)^\top G\frac{\Gamma_0}{N}\tilde{x}_i(0)\left.+\left(x_0(0)-\bar{x}_i^{(N)}(0)\right)^\top \Gamma_0^\top G\left(I_n-\frac{\Gamma_0}{N}\right)\tilde{x}_i(0)\right]\\
            &=\mathbb{E}\Bigg[ \int_0^T\left(-\left(\bar{x}_i-\Gamma_1 \bar{x}_i^{(N)}-\eta_1\right)^\top Q\frac{\Gamma_1}{N} \tilde{x}_i+ \left(\bar{x}_i^{(N)}-x_0\right)\Gamma_1 Q \tilde{x}_i \right.\\
		\end{aligned}
	$$
    $$
		\begin{aligned}		
            &\qquad \left.+\sum_{j=1}^{N} \bar{z}_{i j}^\top S_j \tilde{z}_{i j} +\sum_{j=1}^N \tilde{z}_{i j}^\top \left(\bar{\Pi} \bar{z}_{i j}-\left(S_j+\bar{\Pi}\right)\left(I_n+\bar{\Sigma} S_j\right)^{-1} \bar{\beta}_{i j}\right)\right)dt \\
			&\qquad -\left(\bar{x}_i(0)-\Gamma_0 x_0(0)-\eta_0\right)^\top G\frac{\Gamma_0}{N}\tilde{x}_i(0)\left.+\left(x_0(0)-\bar{x}_i^{(N)}(0)\right)^\top \Gamma_0^\top G\left(I_n-\frac{\Gamma_0}{N}\right)\tilde{x}_i(0)\right].
			\end{aligned}
	$$
	Noticing that
	$$
	\beta_{i j}=\left( I_n+\Sigma S_j\right) z^*_{ij}+\frac{K}{N} S_j \sum_{i=1}^N z^*_{i j},
	$$
	then, we can derive
	$$
    \begin{aligned}
	&\sum_{j=1}^N \bar{z}_{i j}^\top S_j \tilde{z}_{i j} +\sum_{j=1}^N \tilde{z}_{i j}^\top\left(\bar{\Pi} \bar{z}_{ij}-\left(S_j+\bar{\Pi}\right)\left(I_n+\bar{\Sigma} S_j\right)^{-1} \bar{\beta}_{i j}\right)  \\
	&=\sum_{j=1}^N \tilde{z}_{i j}^\top \left(S_j+\bar{\Pi}\right) \bar{z}_{ij}-\sum_{j=1}^N \tilde{z}_{ij}^\top\left(S_j+\bar{\Pi}\right)\left(I_n+\bar{\Sigma} S_j\right)^{-1} \left(\bar{\beta}_{i j}-\beta_{i j}\right)\\
	&\quad -\sum_{j=1}^N \tilde{z}_{ij}^\top \left(S_j+\bar{\Pi}\right)\left(I_n+\bar{\Sigma} S_j\right)^{-1} \left[\left( I_n+\Sigma S_j\right) z^*_{ij}+\frac{K}{N} S_j \sum_{i=1}^N z^*_{i j }\right]\\
	&=\sum_{j=1}^N \tilde{z}_{i j}^\top \left(S_j+\bar{\Pi}\right) \bar{z}_{ij}-\sum_{j=1}^N \tilde{z}_{ij}^\top \left(S_j+\bar{\Pi}\right)\left(I_n+\bar{\Sigma} S_j\right)^{-1} \left(\bar{\beta}_{i j}-\beta_{i j}\right)\\
	&\quad -\sum_{j=1}^N \tilde{z}_{ij}^\top \left(S_j+\bar{\Pi}\right)\left(I_n+\bar{\Sigma} S_j\right)^{-1} \left[\left( I_n+\bar{\Sigma} S_j\right) z^*_{ij}+\left(\bar{\Sigma}-\Sigma\right) S_jz^*_{ij}+\frac{K}{N} S_j \sum_{i=1}^N z^*_{i j }\right]\\
	&=\sum_{j=1}^N \tilde{z}_{i j}^\top \left(S_j+\bar{\Pi}\right) \left(\bar{z}_{ij} -z^*_{ij}\right)-\sum_{j=1}^N \tilde{z}_{ij}^\top \left(S_j+\bar{\Pi}\right)\left(I_n+\bar{\Sigma} S_j\right)^{-1} \left(\bar{\beta}_{i j}-\beta_{i j}\right)\\
	&\quad -\sum_{j=1}^N \tilde{z}_{ij}^\top \left(S_j+\bar{\Pi}\right) \left[ \left(\bar{\Sigma}-\Sigma\right) S_jz^*_{ij}+\frac{K}{N} S_j \sum_{i=1}^N z^*_{i j }\right].
    \end{aligned}   
	$$
	Substitute the above equation into the equation of $\mathcal{I}_i$, we can finally obtain
	$$
	\begin{aligned}
		\mathcal{I}_i&=\mathbb{E}\Bigg[ \int_0^T\Bigg[-\left(\bar{x}_i-\Gamma_1 \bar{x}_i^{(N)}-\eta_1\right)^\top Q\frac{\Gamma_1}{N} \tilde{x}_i+ \left(\bar{x}_i^{(N)}-x_0\right)\Gamma_1 Q \tilde{x}_i \\
	    &\qquad +\sum_{j=1}^N \tilde{z}_{i j}^{\top} \left(S_j+\bar{\Pi}\right) \left(\bar{z}_{ij} -z^*_{ij}\right)-\sum_{j=1}^N \tilde{z}_{ij}^\top\left(S_j+\bar{\Pi}\right)\left(I_n+\bar{\Sigma} S_j\right)^{-1} \left(\bar{\beta}_{i j}-\beta^*_{i j}\right)\\
	    &\qquad -\sum_{j=1}^N\tilde{z}_{ij}^\top\left(S_j+\bar{\Pi}\right) \left[ \left(\bar{\Sigma}-\Sigma\right) S_jz^*_{ij}+\frac{K}{N} S_j \sum_{i=1}^N z^*_{i j }\right]\Bigg]dt\\
	    &\qquad -\left(\bar{x}_i(0)-\Gamma_0 x_0(0)-\eta_0\right)^\top G\frac{\Gamma_0}{N}\tilde{x}_i(0)+\left(x_0(0)-\bar{x}_i^{(N)}(0)\right)^\top \Gamma_0^\top G\left(I_n-\frac{\Gamma_0}{N}\right)\tilde{x}_i(0)\Bigg]\\
	    &=O\left(\frac{1}{\sqrt{N}}\right).
	    \end{aligned}
	    $$
	Therefore, combining (\ref{asymptotic property}) we have
	$$
	\mathcal{J}_i(\bar{u}_i(\cdot);\bar{u}_{-i}(\cdot)) \leqslant \mathcal{J}_i(u_i(\cdot);\bar{u}_{-i}(\cdot))+O \left(\frac{1}{\sqrt{N}}\right).
	$$
	Thus, $(\bar{u}_1(\cdot),\ldots,\bar{u}_N(\cdot))$ is an $\epsilon$-Nash equilibrium.
\end{proof}

\section{Numerical examples}

In this section, we give a numerical example with certain particular coefficients to simulate our theoretical results. We set the number of agents to 300, i.e., $N=300$ and the terminal time is $1$. The simulation parameters are given as follows: $A=0.1,B=2,Q=1,R=5,G=2,\Gamma_1=0.5,\eta_1=1,\Gamma_0=1,\eta_0=1$. And for $i=1,\ldots,N$, $S_i=1,\xi_i=W_i(T)$. By the  Euler's method, we plot the solution curves of Riccati equations (\ref{Sigma limit equation}), (\ref{K limit equation}), (\ref{Pi limit equation}) and (\ref{M limit equation})
in Figure \ref{fig:Riccati}. By the Monte Carlo method, the figures of $\bar{\zeta}_i(\cdot)$ and optimal state $\bar{x}_i(\cdot)$ are shown in Figure \ref{fig:bar_zeta_i} and Figure \ref{fig:bar_x_i}, respectively. Further, we also generate the dynamic simulation of optimal decentralized control $\bar{u}_i(\cdot)$, shown in Figure \ref{fig:bar_u_i}.

\begin{figure}[htbp] 
	\centering  
	\begin{minipage}{0.45\textwidth} 
		\centering  
		\includegraphics[width=\textwidth]{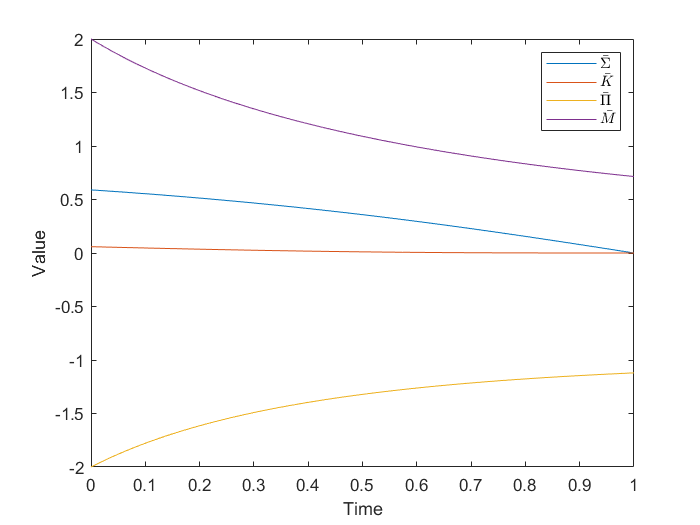} 
		\caption{\centering The solution curve of $\bar{\Sigma}(\cdot)$, $\bar{K}(\cdot)$, $\bar{\Pi}(\cdot)$ and $\bar{M}(\cdot)$} 
		\label{fig:Riccati} 
	\end{minipage}  
	\hspace{0.05\textwidth} 
	\begin{minipage}{0.45\textwidth}  
		\centering  
		\includegraphics[width=\textwidth]{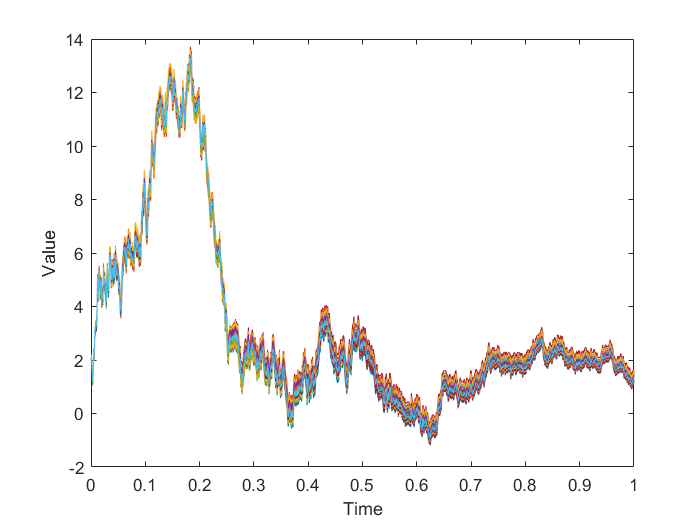} 
		\caption{\centering The solution curve of $\bar{\zeta}_i(\cdot), i=1,\cdots,300$} 
		\label{fig:bar_zeta_i} 
	\end{minipage}  
	\end{figure}  

\begin{figure}[htbp] 
	\centering  
	\begin{minipage}{0.45\textwidth} 
		\centering  
		\includegraphics[width=\textwidth]{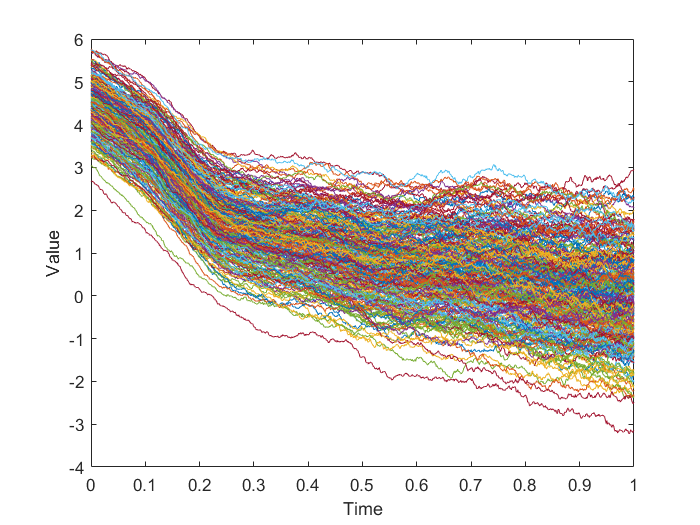} 
		\caption{\centering The solution curve of $\bar{x}_i(\cdot), i=1,\cdots,300$}  
		\label{fig:bar_x_i} 
	\end{minipage}  
	\hspace{0.05\textwidth} 
	\begin{minipage}{0.45\textwidth}  
		\centering  
		\includegraphics[width=\textwidth]{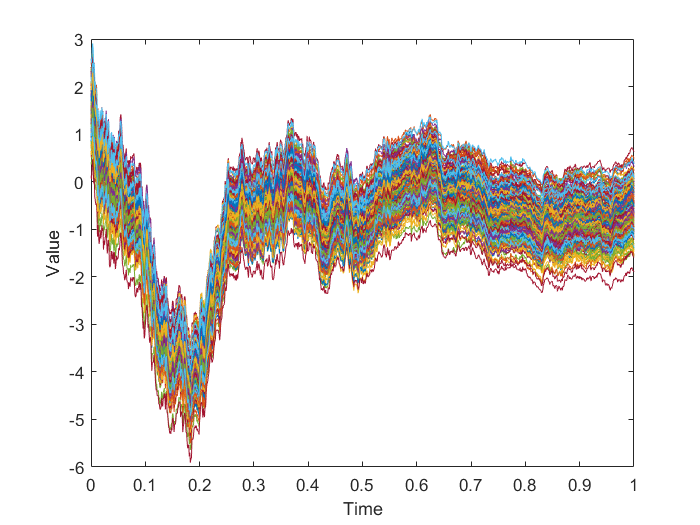} 
		\caption{\centering The solution curve of $\bar{u}_i(\cdot), i=1,\cdots,300$}  
		\label{fig:bar_u_i}   
	\end{minipage}  
\end{figure}  

\section{Conclusion}

In this paper, we have studied the dynamic optimization of large-population system with linear BSDEs. We adopts a direct approach to solve this large-population problem and obtain the decentralized strategy. Our present work suggests various future research directions. For example, (i) To study the backward MFG with indefinite control weight (this will formulate the mean-variance analysis with relative performance in our setting); (ii) To study the backward MFG with integral-quadratic constraint, we can attempt to adopt the method of Lagrange multipliers and the Ekeland variational method; (iii) To consider the direct method to solve mean field problem with the state equation contains state average term. We plan to study these issues in our future works.


\begin{thebibliography}{0}

\bibitem{Bardi-Priuli-2014}M. Bardi and F. Priuli, Linear-quadratic $N$-person and mean-field games with ergodic cost, {\it SIAM J. Control Optim.}, 52, 3022-3052, 2014.

\bibitem{Bensoussan-Feng-Huang-2021}A. Bensoussan, X. Feng and J. Huang, Linear-quadratic-Gaussian mean-field-game with partial observation and common noise, {\it Math. Control Relat. Fields}, 11, 23-46, 2021.

\bibitem{Bensoussan-Frehse-Yam-2013}A. Bensoussan, J. Frehse and P. Yam, {\it Mean Field Games and Mean Field Type Control Theory}, SpringerBriefs in Mathematics, Springer, New York, 2013.

\bibitem{Bismut-1978}J. Bismut, An introductory approach to duality in optimal stochastic control, {\it SIAM Rev.}, 20, 62-78, 1978.

\bibitem{Buckdahn-Djehiche-Li-Peng-2009}R. Buckdahn, B. Djehiche, J. Li and S. Peng, Mean field backward stochastic differential equations: A limit approach, {\it Ann. Probab.}, 37, 1524-1565, 2009.

\bibitem{Carmona-Delarue-2013}R. Carmona and F. Delarue, Probabilistic analysis of mean field games. {\it SIAM J. Control Optim.}, 51, 2705-2734, 2013.

\bibitem{Cong-Shi-2024}W. Cong and J. Shi. Direct approach of linear-quadratic Stackelberg mean field games of backward-forward stochastic systems, {\it arXiv:2401.15835}.

\bibitem{Du-Huang-Wu-2018}K. Du, J. Huang and Z. Wu, Linear-quadratic mean-field-game of backward stochastic differential systems, {\it Math. Control Relat. Fields}, 8, 653-678, 2018.

\bibitem{Feng-Huang-Wang-2021}X. Feng, J. Huang and S. Wang, Social optima of backward linear-quadratic-Gaussian mean-field teams, {\it Appl. Math. Optim.}, 84, 651–694, 2021.

\bibitem{Feng-Wang-2024}X. Feng and L. Wang, Stackelberg equilibrium with social optima in linear-quadratic-Gaussian mean-field system, {\it Math. Control Relat. Fields}, 14, 769-799, 2024.

\bibitem{Hu-Huang-Li-2018}Y. Hu, J. Huang and X. Li, Linear quadratic mean field game with control input constraint, {\it ESAIM Control Optim. Calc. Var.}, 24, 901-919, 2018.

\bibitem{Hu-Huang-Nie-2018}Y. Hu, J. Huang and T. Nie, Linear-quadratic-Gaussian mixed mean-field games with heterogeneous input constraints, {\it SIAM J. Control Optim.}, 56, 2835-2877, 2018.

\bibitem{Huang-Huang-2017}J. Huang and M. Huang, Robust mean field linear-quadratic-Gaussian games with unknown $L^2$-disturbance, {\it SIAM J. Control Optim.}, 55, 2811-2840, 2017.

\bibitem{Huang-Wang-2016}J. Huang and S. Wang, Dynamic optimization of large-population systems with partial information, {\it J. Optim. Theory Appl.}, 168, 231-245, 2016.

\bibitem{Huang-Wang-Wu-2016}J. Huang, S. Wang and Z. Wu, Backward mean-field linear-quadratic-Gaussian (LQG) games: Full and partial information, {\it IEEE Trans. Automat. Control}, 61, 3784-3796, 2016.

\bibitem{Huang-2010}M. Huang, Large-population LQG games involving a major player: The Nash certainty equivalence principle, {\it SIAM J. Control Optim.}, 48, 3318-3353, 2010.

\bibitem{Huang-Caines-Malhame-2007}M. Huang, P. Caines and R. Malham\'e, Large-population cost-coupled LQG problems with nonuniform agents: Individual-mass behavior and decentralized $\varepsilon$-Nash equilibria, {\it IEEE Trans. Automat. Control}, 52,  1560-1571, 2007.

\bibitem{Huang-Malhame-Caines-2006}M. Huang, R. Malham\'e and P. Caines, Large population stochastic dynamic games: Closed-loop McKean-Vlasov systems and the Nash certainty equivalence principle, {\it Commun. Inf. Syst.}, 6, 221-251, 2006.

\bibitem{Huang-Zhou-2020}M. Huang and M. Zhou, Linear quadratic mean field games: Asymptotic solvability and relation to the fixed point approach, {\it IEEE Trans. Automat. Control}, 65, 1397-1412, 2020.

\bibitem{Huang-Wang-Wang-Wang-2023}P. Huang, G. Wang, W. Wang and Y. Wang,  A linear-quadratic mean-field game of backward stochastic differential equation with partial information and common noise, {\it Appl. Math. Comput.}, 446, 127899, 2023.

\bibitem{Khalil-2002}H. Khalil, {\it Nonlinear Systems}, 3rd edition, Prentice Hall, Inc., 2002.

\bibitem{Lasry-Lions-2007}J. Lasry and P. Lions, Mean field games, {\it Jpn. J. Math.}, 2, 229-260, 2007.

\bibitem{Li-Wu-2023}M. Li and Z. Wu, Backward Linear-Quadratic mean field social optima with partial information, {\it Commun. Math. Stat.}, https://doi.org/10.1007/s40304-023-00348-4

\bibitem{Li-Sun-Xiong-2019}X. Li, J. Sun and J. Xiong, Linear quadratic optimal control problems for mean-field backward stochastic differential equations, {\it Appl. Math. Optim.}, 80, 223-250, 2019.

\bibitem{Lim-Zhou-2001}A. Lim and X. Zhou, Linear-quadratic control of backward stochastic differential equations, {\it SIAM J. Control Optim.} 40, 450-474, 2001.

\bibitem{Moon-Basar-2017}J. Moon and T. Ba\c sar, Linear quadratic risk-sensitive and robust mean field games, {\it IEEE Trans. Automat. Control}, 62, 1062-1077, 2017.

\bibitem{Nguyen-Huang-2012}S. L. Nguyen and M. Huang, Linear-quadratic-Gaussian mixed games with continuum-parametrized minor players, {\it SIAM J. Control Optim.}, 50, 2907-2937, 2012.

\bibitem{Pardoux-Peng-1990}E. Pardoux and S. Peng, Adapted solution of a backward stochastic differential equation, {\it Syst. $\&$ Control Lett.}, 14, 55-61, 1990.

\bibitem{Reid-1972}W. Reid, {\it Riccati Differential Equations}, Academic Press, Cambridge, 1972.

\bibitem{Wang-2024}B. Wang, Leader-follower mean field LQ games: A direct method, {\it Asian J. Control}, 26, 617-625, 2024.

\bibitem{Wang-Zhang-Zhang-2022}B. Wang, H. Zhang and J. Zhang, Linear quadratic mean field social control with common noise: A directly decoupling method, {\it Automatica.}, 146, 110619, 2022.

\bibitem{Wang-Zhang-2017}B. Wang and J. Zhang, Social optima in mean field linear-quadratic-Gaussian models with Markov jump parameters, {\it SIAM J. Control Optim.}, 55, 429-456, 2017.

\bibitem{Yong-Zhou-1999}J. Yong and X. Zhou, {\it Stochastic Controls: Hamiltonian Systems and HJB Equations}, Springer-Verlag, New York, 1999.

\end{thebibliography}
\end{document}